\documentclass[a4paper]{article}
\usepackage[utf8]{inputenc}

\usepackage[top=2.5cm, bottom=2.5cm, left=2.5cm, right=2.5cm]{geometry}

\usepackage{amsmath, amssymb, mathtools}
\usepackage[hidelinks]{hyperref}

 \usepackage{xcolor}
 \usepackage{amsthm}
 \usepackage[noabbrev]{cleveref}
 \newtheorem{thm}{Theorem}%[section]
 \newtheorem{lm}[thm]{Lemma}
 \newtheorem{res}[thm]{Result}
 \newtheorem{crl}[thm]{Corollary}
 \newtheorem{prop}[thm]{Proposition}

 \theoremstyle{definition}

 \newtheorem{df}[thm]{Definition}

 \newcommand{\FF}{\mathbb F}

 \renewcommand{\leq}{\leqslant}
 \renewcommand{\geq}{\geqslant}

 \newcommand{\vspan}[1]{\left \langle #1 \right \rangle}
 
 \newcommand{\sett}[2]{ \left\{ #1 \, \, || \, \, #2 \right \} }

 \newcommand{\floor}[1]{\left \lfloor #1 \right \rfloor}

\DeclareMathOperator{\supp}{supp}

 \DeclareMathOperator{\PG}{PG}
 \newcommand{\pg}{\PG}
 \DeclareMathOperator{\PGL}{PGL}
 \DeclareMathOperator{\GL}{GL}

\makeatletter
\let\@fnsymbol\@arabic  % Redefine \@fnsymbol to use numbers instead of symbols
\makeatother

\title{A note on short minimal codes from subgeometries}
\author{
 Sam Adriaensen
  \thanks{Department of Mathematics and Data Science, Vrije Universiteit Brussel, Pleinlaan 2, 1050 Elsene, Belgium. \href{mailto:sam.adriaensen@vub.be}{sam.adriaensen@vub.be}} 
 \and Peter Sziklai
  \thanks{Eotvos L. University, Budapest, Department of Computer Science,
  \href{mailto:peter.sziklai@ttk.elte.hu}{peter.sziklai@ttk.elte.hu}
  }
  \thanks{Department of Mathematics, J. Selye University, 945 01 Komárno, Slovakia}
 \and Zsuzsa Weiner
\thanks{Eotvos L. University, Budapest, Department of Computer Science,
   \href{mailto:zsuzsa.weiner@gmail.com}{zsuzsa.weiner@gmail.com}} 
}
\date{}

\begin{document}

\maketitle

\begin{abstract}
 In a 2022, Bartoli, Cossidente, Marino, and Pavese proved that in the projective space ${\rm PG}(3,q^3)$, one can find three $\mathbb F_q$-subgeometries such that the union of their point sets is a strong blocking set.
 This proves the existence of linear minimal codes with parameters $[3(q^2+1)(q+1),4]_{q^3}$ for every prime power $q$.
 We give a short proof of this result for odd values of $q > 9$, using the theory of small blocking sets in projective planes.
\end{abstract}



\section{Introduction}

Throughout the article, $p$ denotes a prime number, $q = p^h$ denotes a prime power, and the finite field containing $q$ elements will be denoted as $\FF_q$.
The set of non-zero elements of $\FF_q$ is denoted as $\FF_q^*$.

A $k$-dimensional subspace $C$ of $\FF_q^n$ is called a \emph{linear code} with parameters  $[n,k]_q$.
Linear codes over finite fields play a central role in the theory of error-correction, and have  applications in other branches of the mathematics of communication.
Given a codeword $c \in C$, we define its \emph{support} as the set $\supp(c)$ of coordinate positions where $c$ has a non-zero entry.
The size of $\supp(c)$ is called the \emph{(Hamming) weight} of $c$.
A codeword $c \in C$ is called \emph{minimal} if the only codewords $w \in C$ with $\supp(w) \subseteq \supp(c)$ are the scalar multiples of $c$.
The concept of minimal codewords gained popularity after Massey \cite{Massey} introduced a secret sharing scheme, constructed from a linear code $C$, where the minimal access structures correspond to the minimal codewords of $C$.
It is in general a difficult problem to determine which codewords of a linear code are minimal.
This motivated the study of linear codes in which all codewords are minimal, called \emph{minimal codes}.
Arguably, the central problem in the study of minimal codes is to study the shortest possible length $n$ for which a minimal $[n,k]_q$ code exists, in function of $k$ and $q$.
Let us denote this function as $m(k,q)$.

A major breakthrough in the study of $m(k,q)$ occurred quite recently by studying equivalent geometrical structures.
We refer the reader to the recent survey by Scotti \cite{Scotti} for a more detailed account.
Let $\pg(k-1,q)$ denote the projective geometry consisting of the subspaces of the vector space $\FF_q^k$.
The projective dimension of a subspace of $\pg(k-1,q)$ is one less than its vector space dimension.
In particular, the points, lines, and planes of $\pg(k-1,q)$ are the subspaces of $\FF_q^k$ of respective dimensions $1$, $2$, and $3$.
The subspaces of $\pg(k-1,q)$ of (projective) dimension $k-2$ are called \emph{hyperplanes}.

\begin{df}
 A set $B$ of points in $\pg(k-1,q)$ is called a \emph{blocking set} if each hyperplane of $\pg(k-1,q)$ contains a point of $B$.
 If in addition each hyperplane is spanned by its intersection with $B$, we call $B$ a \emph{strong blocking set}.
\end{df}

We refer the reader to \cite{Blokhuis:Sziklai:Szonyi} for a survey on blocking sets.
We mention that the smallest blocking sets of $\pg(k-1,q)$ have size $q+1$ and consist of the points on a line, see e.g.\ \cite[Theorem 2.1]{Blokhuis:Sziklai:Szonyi}.

Strong blocking sets were first defined by Davydov, Giulieti, Marcugini, and Pambianco \cite{Davydov:Etal} in the study of so-called covering codes.
They were introduced into the study of minimal codes by Bonini and Borello \cite{Bonini:Borello} (actually, they were reinvented in \cite{Bonini:Borello} under the name cutting blocking sets).
Alfarano, Borello and Neri \cite{Alfarano:Borello:Neri}, and independently Tang, Qiu, Liao, and Zhou \cite{Tang} proved an equivalence between strong blocking sets and minimal codes.
This equivalence works as follows.
A \emph{generator matrix} of a linear $[n,k]_q$ code $C$ is a matrix $G \in \FF_q^{k \times n}$ whose row space equals $C$.
We will assume that $G$ has no zero columns, such that each column of $G$ corresponds to a point of $\pg(k-1,q)$.
This yields a (multi)set of $n$ points in $\pg(k-1,q)$, which we call a \emph{projective system} associated to $C$.
Note that if $G$ is one generator matrix of $C$, then its other generator matrices are $M G$ with $M \in \GL(k,q)$.
Therefore, the projective systems associated to $C$ form an orbit under the action of $\PGL(k,q)$.

\begin{res}[{\cite{Alfarano:Borello:Neri,Tang}}]
 A linear $[n,k]_q$ code is minimal if and only if its associated projective systems are strong blocking sets.
\end{res}

Therefore, $m(k,q)$ can equivalently be defined as the size of the smallest strong blocking set in $\pg(k-1,q)$.
This interpretation allowed Alfarano, Borello, Neri, and Ravagnani \cite[Theorem 2.14]{Alfarano:Borello:Neri:Ravagnani} to prove that $m(k,q) \geq (k-1)(q+1)$.
On the other hand, H\'eger and Nagy \cite[Theorem 4.1]{Heger:Nagy} proved that $m(k,q) \leq (2-o_q(1))(k-1)(q+1)$ using a probabilistic construction.
Therefore, we are interested in constructing families of strong blocking sets in $\pg(k-1,q)$ whose size is along the lines of $c (k-1) q$ for a reasonably small constant $c$.

Several of the geometric constructions of strong blocking sets consist of taking the union of several small blocking sets.
Recall that the lines constitute the smallest blocking sets.
Sets of lines whose union yields a strong blocking set are said to be \emph{in higgledy-piggledy arrangement}, and their study was initiated by Fancsali and the second author \cite{Fancsali:Sziklai}, actually predating the link with minimal codes.
Subsequently, several papers have investigated small sets of lines in $\pg(k-1,q)$ whose union yields a strong blocking set, see e.g.\ \cite{Alon:Etal,Bartoli:Cossidente:Marino:Pavese,Denaux} and the aforementioned result by H\'eger and Nagy \cite{Heger:Nagy}.
If $q \geq \floor{\frac32 (k-1)}$, then a set of lines in higgledy-piggledy arrangement in $\pg(k-1,q)$ contains at least $\floor{\frac32 (k-1)}$ lines \cite{Fancsali:Sziklai}.
The corresponding strong blocking set therefore has size approximately at least $\frac 32 (k-1)q$.
It is hence natural to wonder whether we can lower the constant $\frac 32$ by taking unions of other small blocking sets than lines.

\begin{df}
 Consider in $\pg(k-1,q^m)$ the subspaces of $\FF_{q^m}^k$ having a basis of vectors of $\FF_q^k$.
 These subspaces form a projective geometry isomorphic to $\pg(k-1,q)$.
 This set of subspaces and any image of this set under the action of $\PGL(k,q^m)$ is called a \emph{subgeometry} of $\pg(k-1,q^m)$ (or an $\FF_q$-subgeometry in case we want to emphasize which subfield of $\FF_{q^m}$ is used to construct the subgeometry).
\end{df}

The points of an $\FF_q$-subgeometry of $\pg(k-1,q^{k-1})$ are known to be a blocking set of size $\frac{q^k-1}{q-1}$, see e.g.\ \cite[Proposition 2.10]{Blokhuis:Sziklai:Szonyi}.
The size is close to $q^{k-1}+1$, the number of points on a line.
Moreover, the points of $\pg(k-1,q^{k-1})$ can be partitioned into disjoint $\FF_q$-subgeometries, see e.g.\ \cite[\S 4.3]{Hirschfeld:98}.
In this note, we will explore how to choose a small number of subgeometries, such that the union of their point sets constitutes a strong blocking set.
In $\pg(k-1,q^{k-1})$, one would need at least $k-1$ subgeometries.
In case these $k-1$ subgeometries can be found, this yields a strong blocking set of size $(k-1) \frac{q^k-1}{q-1}$, whose size is quite close to lower bound $(k-1)(q^{k-1}+1)$.

Let us first examine the case $k=3$.
A strong blocking set in $\pg(2,q)$ is just a set of points intersecting every line in at least 2 points.
Therefore, the union of the points of two disjoint $\FF_q$-subgeometries of $\pg(2,q^2)$ forms a strong blocking set.
It follows that $m(3,q^2) \leq 2 (q^2 + q + 1)$.
This bound was proven to be tight if $q \geq 5$ by Ball and Blokhuis \cite{Ball:Blokhuis}.

The case $k = 4$ is more complicated.
The union of the points of 3 disjoint $\FF_q$-subgeometries of $\pg(3,q^3)$ intersects every plane in at least 3 points, but does not necessarily form a strong blocking set.
Indeed, to ensure that these subgeometries constitute a strong blocking set, they need to be chosen in a more careful manner.
However, Bartoli, Cossidente, Marino and Pavese \cite{Bartoli:Cossidente:Marino:Pavese} proved that you can always find 3 suitable subgeometries.

\begin{res}[{\cite[Theorem 2.26]{Bartoli:Cossidente:Marino:Pavese}}]
 \label{Res:BCMP}
 In $\pg(3,q^3)$, there exist 3 $\FF_q$-subgeometries, such that the union of their point sets is a strong blocking set of size $3(q^2+1)(q+1)$.
\end{res}

The proof of this beautiful result as presented in \cite{Bartoli:Cossidente:Marino:Pavese} is quite long.
The goal of this note is to use the literature on blocking sets in $\pg(2,q)$ to give a considerably shorter proof, at the expense that we need the extra condition that $q = 7$ or $q > 9$ is odd.
We hope that our results can be extended to prove that there exist $k-1$ subgeometries in $\pg(k-1,q^{k-1})$ such that the union of their point sets forms a strong blocking set for higher values of $k$.
To this end, we describe the setup of the problem for general values of $k$ in \Cref{Sec:General k}.
We focus on the case $k=4$ and give our alternative proof for \Cref{Res:BCMP} (assuming that $q \geq 7$) in \Cref{Sec:k=4}.

\section{Strong blocking sets from disjoint subgeometries}
 \label{Sec:General k}

Let us review how the points of $\pg(k-1,q^{k-1})$ can be partitioned into the points of $\FF_q$-subgeometries, we refer the reader to \cite[\S 4.3]{Hirschfeld:98} for more details.
To construct $\pg(k-1,q^{k-1})$, we need a $k$-dimensional vector space over $\FF_{q^{k-1}}$.
We can choose the field $\FF_{q^{k(k-1)}}$ to be this vector space.
Define
\[
 R = \FF_{q^{k-1}} \cdot \FF_{q^k} \coloneqq \sett{x y}{x \in \FF_{q^{k-1}}, \, y \in \FF_{q^{k}}},
\]
and let $R^*$ be $R \setminus \{0\}$.
Then $R^*$ is a multiplicative subgroup of $\FF_{q^{k(k-1)}}^*$.
Consider a coset $\alpha R^*$ of $R^*$.
Consider the set of $1$-dimensional $\FF_{q^{k-1}}$-subspaces, each of them being spanned by an element of $\alpha R^*$.
This can be seen as a set of points in $\pg(k-1,q^{k-1})$, and as such constitutes the set of points of an $\FF_q$-subgeometry.
Since the cosets of $R^*$ partition $\FF_{q^{k(k-1)}}^*$, this partitions the points of $\pg(k-1,q^{k-1})$ into point sets of subgeometries.

\begin{df}
 Call a set of elements $\alpha_1, \dots, \alpha_m$ in $\FF_{q^{k(k-1)}}$ \emph{$R$-independent} if the only solution to $\rho_1 \alpha_1 + \ldots + \rho_m \alpha_m = 0$ with all $\rho_i \in R$ is $\rho_1 = \ldots = \rho_m = 0$.
\end{df}

\begin{lm}
 \label{Lm:Independent}
Let $\alpha_1, \dots, \alpha_{k-1}$ be elements of $\FF_{q^{k(k-1)}}$.
Consider the set $S$ of points in $\pg(k-1,q^{k-1})$ having a coordinate vector in $\bigcup_{i=1}^{k-1} \alpha_i R^*$.
If $\alpha_1, \dots, \alpha_{k-1}$ are $R$-independent, then $S$ is a strong blocking set.
Furthermore, if $q \geq k$, then the converse holds.
\end{lm}

\begin{proof}
Take a hyperplane $\Pi$ in $\pg(k-1,q^{k-1})$.
For every $i$, the points of the subgeometry corresponding to $\alpha_i R^*$ form a blocking set, hence contain at least one point $P_i$ of $\Pi$.
Then $P_i$ has a coordinate vector $\rho_i \alpha_i$ for some $\rho_i \in R^*$.
We claim that $\rho_1 \alpha_1, \dots, \rho_{k-1} \alpha_{k-1}$ are $\FF_{q^{k-1}}$-linearly independent.
Indeed, suppose that $\sigma_1, \dots, \sigma_{k-1}$ 
are scalars in $\FF_{q^{k-1}}$ such that $\sum_{i=1}^{k-1} \sigma_i \rho_i \alpha_i = 0$.
Then each $\sigma_i \rho_i$ is an element of $R$, and since the $\alpha_i$ are $R$-independent, all $\sigma_i \rho_i$ are zero.
Since none of the $\rho_i$ are zero, this means that all the $\sigma_i$ are zero.
Therefore, the points $P_i$ are independent and must span $\Pi$.

Now suppose that $q \geq k$, and that $\alpha_1, \dots, \alpha_{k-1}$ are not $R$-independent.
Then there exists a $(k-3)$-space $\pi$ of $\pg(k-1,q^{k-1})$ containing a point of each subgeometry.
Suppose that the subgeometry corresponding to $\alpha_i R^*$ intersects $\pi$ in a subspace of dimension $m_i-1$.
Then $m_i \geq 1$, and there are 
\[
 \frac{q^{k-m_i}-1}{q-1} \leq \frac{q^{k-1}-1}{q-1}
\]
$m_i$-spaces of the subgeometry through its intersection with $\pi$.
Therefore, there are at most
\[
 (k-1) \frac{q^{k-1}-1}{q-1}
\]
hyperplanes through $\pi$ that contain another point of $S$.
If $q \geq k$, this number at most $q^{k-1}-1$, which is smaller than $q^{k-1}+1$, the number of hyperplanes through $\pi$.
Hence, there exists a hyperplane $\Pi$ through $\pi$ such that $\Pi \cap S$ is contained in $\pi$.
Thus, $S$ is not a strong blocking set.
\end{proof}

\begin{df}
Define $B(k,q)$ to be the set of points in $\pg(k-2,q^{k(k-1)})$ having a coordinate vector in $R^{k-1} \subset \FF_{q^{k(k-1)}}^{k-1}$.
\end{df}

\begin{prop}
 \label{Prop:Not blocking}
 If $B(k,q)$ is {\bf not} a blocking set of $\pg(k-2,q^{k(k-1)})$, then $\pg(k-1,q^{k-1})$ contains $k-1$ subgeometries such that the union of their point sets is a strong blocking set.
\end{prop}

\begin{proof}
 Suppose that $B(k,q)$ is not a blocking set.
 Then there exists a hyperplane, defined by some linear equation $\alpha_1 X_1 + \ldots + \alpha_{k-1} X_{k-1} = 0$, that does not intersect $B(k,q)$.
 This is equivalent to $\alpha_1, \dots, \alpha_{k-1}$ being $R$-independent.
 The proposition then follows from \Cref{Lm:Independent}.
\end{proof}

Thus, our goal is to prove that $B(k,q)$ is not a blocking set.
We note that
\begin{align*}
 r \coloneqq |R^*| = \frac{(q^{k-1}-1)(q^k-1)}{q-1},
 && \text{and} &&
 |B(k,q)| = \frac{(r+1)^{k-1} - 1}r.
\end{align*}

Next, we take a look at the symmetries of $B(k,q)$.
Consider a projective space $\pg(k-1,q)$, and suppose that its underlying vector space is the standard vector space $\FF_q^k$.
Given a point $P$ of $\pg(k-1,q)$, we define the \emph{weight} of $P$ as the weight of any of its coordinate vectors, i.e.\ the number of coordinate positions with non-zero entry.

\begin{lm}
 \label{Lm:Stabiliser}
 The setwise stabiliser of $B(k,q)$ in $\PGL(k-1,q^{k(k-1)})$ acts transitively on the points of $B(k,q)$ of fixed weight.
\end{lm}

\begin{proof}
 It suffices to prove that the setwise stabiliser of $R^{k-1}$ in $\GL(k-1,q^{k(k-1)})$ acts transitively on the vectors of $R^{k-1}$ of fixed weight.
 Consider a permutation $\sigma$ in the symmetric group on $k-1$ elements, and choose $\rho_1, \dots, \rho_{k-1} \in R^*$.
 Consider the linear map $f \in \GL(k-1,q^{k(k-1)})$ defined by
 \[
  f: (v_1, \dots, v_{k-1}) \mapsto (\rho_1 v_{\sigma(1)}, \dots, \rho_{k-1} v_{\sigma(k-1)}).
 \]
 Then $f$ fixes $R^{k-1}$ setwise.
 The set of all maps $f$ of this form constitute a subgroup $G$ of $\GL(k-1,q^{k(k-1)})$.
 One easily checks that the orbits of $G$ on the vectors of $R^{k-1}$ partition the vectors according to their weight.
\end{proof}

\begin{df}
 Let $B$ be a set of points in a projective space.
 A line $\ell$ is said to be \emph{tangent} to $B$ if $|B \cap \ell| = 1$, and \emph{secant} if $|B \cap \ell| > 1$.
\end{df}

\begin{lm}
 \label{Lm:Weight 1}
 Consider a point $P$ of weight one in $\pg(k-2,q^{k(k-1)})$.
 Then there are $\frac{(r+1)^{k-2}-1}r$ secant lines to $B(k,q)$ through $P$, all intersecting $B(k,q)$ in $r+2$ points.
 The other lines through $P$ are tangent to $B(k,q)$.
\end{lm}

\begin{proof}
 By \Cref{Lm:Stabiliser}, we may suppose that $P = \vspan{(1,0,\dots,0)}$.
 Take a line $\ell$ through $P$.
 Suppose that $\ell$ contains another point $Q$ of $B(k,q)$.
 Then $Q$ has a coordinate vector $(\rho_1, \dots, \rho_{k-1}) \in R^{k-1}$.
 The line $\ell$ intersects $B(k,q)$ exactly in the points $P$ and $\vspan{(\rho_1', \rho_2, \dots, \rho_{k-1})}$ with $\rho_1' \in R$.
 Hence, $\ell$ contains $r+2$ points of $B(k,q)$.
 The number of such lines must equal
 \[
  \frac{|B(k,q)|-1}{r+1} = \frac{(r+1)^{k-2}-1}r. \qedhere
 \]
\end{proof}

\section{Reproving Result \texorpdfstring{\ref{Res:BCMP}}{}}
 \label{Sec:k=4}

In this section, we will give an alternative proof for \Cref{Res:BCMP} in case that $q$ is odd, and $q \notin \{3,5,9\}$.
By \Cref{Prop:Not blocking}, it suffices to prove that $B(4,q)$ is not a blocking set of $\pg(2,q^{12})$.
To this end, we recall some results concerning blocking sets in projective planes.

\begin{df}
 Suppose that $B$ is a blocking set of $\pg(2,q)$.
 We call a point $P \in B$ \emph{essential} if $B \setminus \{P\}$ is not a blocking set.
 We call $B$ \emph{minimal} if every point of $B$ is essential.
 We call $B$ a \emph{small} blocking set if $|B| < \frac 32 (q+1)$.
\end{df}

\begin{res}[{\cite[\S 3]{Szonyi:97}}]
 \label{Res:Essential}
 Suppose that $B$ is a blocking set in $\pg(2,q)$ with $|B| \leq 2q$.
 Then $B$ contains a unique minimal blocking set, namely the set of all its essential points.
\end{res}

\begin{res}[{\cite[Corollary 4.8]{Szonyi:97}}]
 \label{Res:Mod}
 Suppose that $B$ is a small minimal blocking set in $\pg(2,q)$, with $q = p^h$ and $p$ prime.
 Then every line intersects $B$ in $1 \pmod p$ points.
\end{res}

\begin{res}
[{\cite[Theorem 5.6] {Szonyi:97}}]
\label{res:interval}
Let $B$ be a minimal non-trivial blocking
set in $\pg(2,q), q = p^n$. Suppose that $|B| < 3(q + 1)/2$. Then
\begin{align}
 \label{eq}
 q + 1 + \frac{q}{p^e+2} \leq |B| \leq \frac{qp^e + 1 -\sqrt{(qp^e + 1)^2 - 4q^2 p^e}}{2}
\end{align}
for some integer $e\geq 1$.
\end{res}

\begin{lm}
 Suppose that $q \geq 7$.
 If $B(4,q)$ is be a blocking set, then it is a small minimal blocking set.
\end{lm}

\begin{proof}
 Suppose that $B(4,q)$ is a blocking set.
 Define $r$ as before to be the size of $R^* = \FF_{q^3}^* \cdot \FF_{q^4}^*$, i.e.\ $r = (q^4-1)(q^2+q+1)$.
 Then $|B(4,q)| = r^2 + 3r + 3$.
 One can verify computationally that $|B(4,q)| < \frac 32 (q^{12}+1)$ if $q \geq 7$, hence $B(4,q)$ must be a small blocking set.

 Next, let $B'$ be the set of essential points of $B(4,q)$.
 By \Cref{Res:Essential}, $B'$ is a small minimal blocking set.
 Note that by \Cref{Lm:Stabiliser}, if one point of a certain weight is essential, then all points of this weight are essential.
 By \Cref{Lm:Weight 1}, points of weight 1 lie on tangent lines to $B(k,q)$, hence are essential.
 If the points of weight 3 of $B(4,q)$ were not essential, then $|B'| \leq 3r + 3 < q^{12}+1$, which is impossible given that $B'$ is a blocking set.
 Lastly, let $P_2$ be a point of weight 2, without loss of generality $(0,1,1)$.
 Let $P_1$ be the point $(1,0,0)$.
 Consider the line $\ell$ through $P_1$ and $P_2$.
 Then $\ell$ intersects $B(4,q)$ in $P_1$, $P_2$, and $r$ other points, all of weight 3, cf.\ the proof of \Cref{Lm:Weight 1}.
 If $P_2$ was not essential, then $B'$ would intersect $\ell$ in $r+1 \equiv 0 \pmod p$ points, contradicting \Cref{Res:Mod}.
 Hence, $B(4,q)$ must be a minimal blocking set.
\end{proof}

Now we prove that for $q\geq 7$, $q \neq 9$, the minimal blocking set $B(4,q)$ does not fit in the intervals in Result \ref{res:interval}, hence a contradiction follows.  

\begin{prop}
    $B(4,q)$ is not a blocking set if $q=7$ or $q>9$ is odd. 
\end{prop}
    
\begin{proof}
Firstly, we prove that $|B(4,q)|=q^{12} + 2q^{11} + 3q^{10} + 2q^9 - q^8 - 4q^7 - 3q^6 - q^5 + 2q^4 + 2q^3 -q + 1$
is greater than the right hand side of (\ref{eq}) in the case that $p^e = q$.

For this, observe that for $q\geq 7$,
$|B(4,q)|\geq q^{12} + 2q^{11} + 3q^{10}$
and 
$$q^{12} + 2q^{11} + 3q^{10} > \frac{q^{13}+1-\sqrt{(q^{13}+1)^2-4q^{25}}}{2}.$$
The latter is true, because
\begin{align*}
 && (q^{13}+1)^2-4q^{25} &> (q^{13}+1 - 2q^{12} -4q^{11} -6q^{10})^2 \\
 \Leftarrow && 0 &> -4q^{24} \Big(1 - \underbrace{\frac1q - \frac{10}{q^2} -\frac{12}{q^3} - \frac{9}{q^4} + \frac{1}{q^{12}} + \frac{2}{q^{13}} + \frac{3}{q^{14}}}_{\text{$< 1$ for $q \geq 7$}
 %|...|< 1\textrm{ if } q\geq 7
 }\Big).
\end{align*}
%$$ (q^{13}+1)^2-4q^{25}> (q^{13}+1 - 2q^{12} -4q^{11} -6q^{10})^2$$
% $$0> -4q^{24} +4q^{23} -40q^{22} +48q^{21} +36q^{20} -4q^{12} -8q^{11} -12q^{10} $$
%$$0> -4q^{24} (1 - \underbrace{{1\over q} - {10\over q^2} -{12\over q^3} -{9\over q^4} +{1\over q^{12}} +{2\over q^{13}} +{3\over q^{14}}).}\atop{|...|< 1\textrm{ if } q\geq 7} $$
Now we prove that $|B(4,q)|$ is less than the left hand side of (\ref{eq}) for the smaller value of $e$, i.e.\ the case where $p^e = q/p$.
We only need to consider this if $q$ is not prime.
Hence, we need to prove that
\[
 q^{12} + 2q^{11} + 3q^{10} + 2q^9 - q^8 - 4q^7 - 3q^6 - q^5 + 2q^4 + 2q^3 -q + 1< q^{12}+1+\frac{q^{12}}{\frac qp +2}.
\]
Estimating the left hand side from above we get
$2q^{11} + 3q^{10} + 2q^9 < \frac{q^{12}}{q/p+2}$, so
\[
 \frac 2p q^{12} + \left(4 + \frac 3p \right) q^{11} + \left( 6 + \frac 2p \right) q^{10} + 4q^9 < q^{12}
\]
which always holds if $p>2$, $(p,q)\neq (3,9)$, and $q \geq p^2$.
\end{proof}

\begin{crl}
    If $q=7$ or $q>9$ odd then in $\pg(3,q^3)$, there exist 3 $\FF_q$-subgeometries, such that the union of their point sets is a strong blocking set of size $3(q^2+1)(q+1)$. \qed
\end{crl}

\paragraph{Acknowledgements.}
We are grateful to Alessandro Neri, Ferdinando Zullo, and Lins Denaux for helpful discussions.
The first author was partially supported by Fonds Wetenschappelijk Onderzoek project 12A3Y25N and by a Fellowship of the Belgian American Educational Foundation. The second author was partially supported by the National Research, Development and Innovation Office within the framework of the Thematic Excellence Program 2021 - National Research Subprogramme: “Artificial intelligence, large networks, data security: mathematical foundation and applications”. The third author was partially supported by the National Research, Development and Innovation Fund, grant number ADVANCED 153080.

\end{document}